\documentclass[12pt,reqno]{amsart}
\usepackage{amssymb,latexsym}
\newtheorem{theorem}{Theorem}
\newtheorem{lemma}{Lemma}
\newtheorem*{theorema} {Theorem A (Wolff)}

\begin{document}

\title [Wolff's Ideal Theorem]
{Wolff's Ideal Theorem on  $Q_p$ Spaces}
\author{Debendra P. Banjade}
\address{Department of Mathematics and Statistics\\
         Coastal Carolina University\\
         Box 261954\\
         Conway, SC  29528-6054\\
         (843)349-6569}
\email{dpbanjade@coastal.edu}
\thanks{Research supported in a part by \lq \lq Professional Activities Travel Grant" from Coastal Carolina University}
\begin{abstract}
For $p\in(0,1),$ let $Q_p$ space be the space of all analytic functions on the unit disk $\mathbb{D}$ such that $\vert f'(z) \vert ^2 (1-\vert z\vert ^2)^p dA(z)$ is a $p$ - Carleson measure. In this paper, we prove that the Wolff's Ideal Theorem on $H^\infty{(\mathbb{D})}$ can be extended to the Banach algebra $H^{\infty}(\mathbb{D})\cap Q_{p}$,  and also to the multiplier algebra on $Q_p$ spaces.
\end{abstract}
\maketitle
\section{Introduction}

For $p\geq{0}$, let $Q_p$ be the space of all analytic functions on the unit disk $\mathbb{D}$ with the norm\\
$$\left\vert\left\vert f \right\vert\right\vert_{Q_{p}}^2:=\int _{-\pi}^\pi \vert f\vert^2 d\sigma+\underset{a\in \mathbb{D}}{\sup}\int_\mathbb{D}\vert f'(z)\vert^2(1-\vert \varphi_{a}(z)\vert^2)^p dA(z)\ < \infty,$$
 where $\varphi_{a}(z)=\frac{a-z}{1-\bar{a}z}$ is a M$\ddot{o}$bius map. It is well known that $Q_{0}=\mathcal{D}$, is the classical Dirichlet space and $Q_{1}=BMOA=BMO( \mathbb{T})\cap H^{2}(\mathbb{D}).$\\ 
The case we are interested in is $p\in (0,1)$.\\

Let $\mathcal{M}(Q_{p})$ be the multiplier algebra of $Q_p$ spaces which we define,  $$\mathcal{M}(Q_{p}):=\left\{ \phi\in Q_{p} :\, M_{\phi}(f)=\phi f\in Q_p\; \text{ for all } f\in Q_{p}\right\}.$$ 
We know from[X1] that $\mathcal{M}(Q_{p})\subseteq H^{\infty}(\mathbb{D}) \cap Q_p$.\\

 In 1962, Carleson [C] proved his famous \lq \lq Corona theorem" characterizing when a finitely generated ideal in
$H^{\infty}(\mathbb{D})$ is actually all of $H^{\infty}(\mathbb{D})$.  Independently, Rosenblum [R], Tolokonnikov [To], and
Uchiyama gave an infinite version of Carleson's work on $H^{\infty}(\mathbb{D})$. In 1997, Nicolau and Xiao [NX] proved that the corona theorem holds for the Banach algebra $H^{\infty}(\mathbb{D})\cap Q_{p}$ and later Xiao [X1] gave a necessary and sufficient condition for the solvability of the corona theorem on $Q_p$ spaces whereas a similar result on  $\mathcal{M}(Q_p)$ was established by Pau [P]. \\

In light of the corona theorem it is natural to ask whether the corona kind of result still holds if we replace the uniform lower bound by any $H^\infty(\mathbb{D})$ function.  Namely, let $f_1, f_2, ...,f_n$ be $H^\infty$ functions, and suppose $g\in\text{H}^\infty$  satisfies 
\maketitle
\begin{equation}
\vert g(z)\vert\leq \vert f_1(z)\vert+...\vert f_n(z)\vert \;\; \text{ for all } \; z\in \mathbb{D}.
\end{equation}\\
Then the question is whether (1) always implies $g\in\mathcal{I}(f_1,...f_n)$, the ideal generated by $\{f_{j}\}_{j=1}^n$. Unfortunately, the answer is no (see [G, p. 369] for an example given by Rao).\\

However, T. Wolff in [G] has proved the following version to show that (1) implies $g^3\in\mathcal{I}(f_1,...f_n).$

%Theorem A (Wolff)
\begin{theorema}
If
\begin{align}
\{f_j\}_{j=1}^n \subset H^{\infty}(\mathbb{D}), g \in H^{\infty}(\mathbb{D}) \quad \text{and}\notag\\
|g(z)| \le \left( \sum_{j=1}^n \, |f_j(z)|^2\right)^{\frac 1{2}} \;\; \text{for all } \, z \in \mathbb{D},
\end{align}
then
\[
g^3 \in \mathcal{I} ( \{ f_j\}_{j=1}^n),
\]

\medskip \noindent
the ideal generated by $\{f_j\}_{j=1}^n$ in $H^{\infty}(\mathbb{D})$.
\end{theorema}
It is also known that (1) is not sufficient for $g^2$  to be in $\mathcal{I} ( \{ f_j\}_{j=1}^n)$ (see Treil [T1]).

\medskip
We have proved  in [BT1] and [BT2] that Wolff's theorem can be extended to the multiplier algebras of Dirichlet and weighted Dirichlet spaces. This paper is devoted to the extension of Wolff's theorem to the Banach algebra $H^{\infty}(\mathbb{D})\cap Q_{p}$ and also to the  multiplier algebra of $Q_p$ spaces. In Dirichlet spaces, using complete Nevanlinna pick kernels, the authors used Hilbert space version directly and then applied abstract operator theory result to establish  the theorem.  But, for  $Q_p$ spaces,  we are unable to use those Hilbert space techniques because these are only Banach spaces for $p\in(0,1)$.  To overcome this difficulty, we will apply  $\bar{\partial}$ - method and some $p$-Carleson measures for  $Q_p$ spaces.\\

%Theorem 1
\begin{theorem}
Let $g$, $\{f_j\}_{j=1}^n \subset H^{\infty}(\mathbb{D}) \cap Q_p$.  Assume that
\begin{equation}
|g(z)| \le \sqrt{\sum_{j=1}^{n}\vert f_j(z)\vert ^2}  \;\; \;      \text{for all} \; \;  {\text{z}}\in\mathbb{D}.
\end{equation} 
Then there exist $\{g_j\}_{j=1}^{n} \subset H^{\infty}(\mathbb{D}) \cap Q_p$ such that
\begin{equation}
\sum_{j=1}^{n}{f_j(z)g_j(z)}=g^3(z)    \;\;\;  \text{for all} \;\;  z\in \mathbb{D}.
\end{equation} 

\end{theorem}
If $p\geq 1$, the Banach algebra $H^{\infty}(\mathbb{D}) \cap Q_p$ is just $H^{\infty}(\mathbb{D})$. Then Theorem 1 definitely holds true, which was the result proved by Wolff in  Theorem A.\\

Before stating our next theorem, it's worthwhile to note that  for $0<p<1$, a Blaschke product B is in $\mathcal{M}({Q_p})$ if and only if it is a finite Blaschke product (see [P]). It simply suggests us that the same Rao's example serves for the counter example in  $\mathcal{M}({Q_p})$ as in $H^\infty(\mathbb{D})$. So (1) is not sufficient for $g\in\mathcal{I}(f_1 ,...f_n)$ in $\mathcal{M}(Q_p)$.

%Theorem 2
\begin{theorem}
Let $g$, $\{f_j\}_{j=1}^n \subset \mathcal{M}(Q_p)$.  Assume that
\begin{equation}
|g(z)| \le \sqrt{\sum_{j=1}^{n}\vert f_j(z)\vert ^2}  \;\; \;      \text{for all} \; \;  {\text{z}}\in\mathbb{D}.
\end{equation}

Then there exist $\{g_j\}_{j=1}^{n} \subset \mathcal{M}(Q_p)$ such that
\begin{equation}
\sum_{j=1}^{n}{f_j(z)g_j(z)}=g^3(z)    \;\;\;  \text{for all} \;\;  z\in \mathbb{D}.
\end{equation} 

\end{theorem}

\medskip
We remark that the cases $p=0$ and $\varphi_{a}(z)=z$ of Theorem 2 were proved in [BT1] and [BT2] for infinite number of generators.\\

The paper is organized as follows: In section 2, we collect some results of $Q_p$ and $Q_p (\mathbb{T})$.  We prove our Theorems 1 and 2 in section 3 and we point out some interesting open problems in section 4 . We use the notation $A \lesssim B$ to indicate that there is a constant $c>0$ with $A\leq cB$ and the notation $A \approx B$ to indicate  $A\lesssim B$ and $B \lesssim A$.
\medskip
\section{Basic properties of $Q_p$ spaces and $Q_p (\mathbb{T})$}
\medskip
\subsection{The p-Carleson measures}\
Given an arc $\mathit{I}$ of the unit circle $\mathbb{T}$ with normalized length $\vert I \vert \leq 1$,  let $$S(\mathit{I}):=\left\{ r{e}^{it}\in\mathbb{D}:1-\vert \mathit{I}\vert<r<1, e^{it}\in \mathit{I}\right\}.$$  $S(\mathit{I})$ is called the Carleson square or a sector based on $\mathit{I}.$ For $0<p<\infty$, we say that a positive Borel measure $\mu$ on $\mathbb{D}$ is a $p$-Carleson measure on $\mathbb{D}$ if $$\vert\vert \mu \vert \vert_{p}=\sup_{\mathit{I}\subset{\mathbb{T}}} \frac{\mu(S(\mathit{I}))}{\vert \mathit{I}\vert^p}<\infty,$$ where the supremum is taken over all subarcs $\mathit{I}$ of $\mathbb{T}$.\\
Equivalently, $\mu$ is a $p$-Carleson measure if and only if there is a constant $C>0$ such that $\mu(S(\mathit{I}))\leq C\vert \mathit{I}\vert^p$ for any subarcs $\mathit{I}$ of $\mathbb{T}$. Also, p-Carleson measures can be described in terms of conformal invariants of those positive measures $\mu$ for which $$\sup_{a\in\mathbb{D}} \int_{\mathbb{D}}\left(\frac{1-\vert a\vert^2}{\vert1-{\bar{a}}z\vert^2}\right)^p d\mu(z)<\infty,$$ and this quantity is equivalent to $\vert\vert\mu\vert\vert_p$\; (see [X2]).\\

The  Lemma 1 can be found in [X2], however,  we provide a proof for the completeness.

\begin{lemma}
Let, $0<p\leq 1$. An analytic function f is in $Q_p$ if and only if the measure $\vert f'(z)\vert ^2 \left(1-\vert z\vert^2\right)^p dA(z)$ is a $p$-Carleson measure.
\end{lemma}
\begin{proof}
Let $f\in{Q_p}$, then by definition of $Q_p$ norm, we have that $$ \underset{a\in \mathbb{D}}{\sup}\int_\mathbb{D}\vert f'(z)\vert^2(1-\vert \varphi_{a}(z)\vert^2)^p dA(z)\ < \infty.$$
This implies that
\begin{align*}
&\underset{a\in \mathbb{D}}{\sup}\int_\mathbb{D}\vert f'(z)\vert^2 \frac{\left(1-\vert z\vert^2\right)^p \left(1-\vert a \vert^2\right)^p}{\vert 1-\bar{a}z\vert^{2p}} dA(z) \\
&=\underset{a\in \mathbb{D}}{\sup}\int_\mathbb{D} \left(\frac{ 1-\vert a \vert^2}{\vert 1-\bar{a}z\vert^{2}}\right)^p \vert f'(z)\vert^2 \left(1-\vert z\vert^2\right)^p dA(z)\ < \infty.
\end{align*}

Hence, by definition of $p$-Carleson measure,  $\vert f'(z)\vert ^2 \left(1-\vert z\vert^2\right)^p dA(z)$ is an $p$-Carleson measure.
Converse can be easily obtained just by reversing the above argument.
\end{proof}
\medskip
We need the following series of lemmas whose proofs are excluded here. One can refer to [P], [NX], [X1], and [X2] for complete proofs. 
\medskip
\begin{lemma}
Let $z\in {\mathbb{D}}, t<-1$ and $c>0$. Then
$$ \int_\mathbb{D}\frac{(1-\vert w\vert ^2)^{t}}{\vert 1-\bar{w}z\vert^{2+t+c}} \approx (1-\vert z\vert ^2)^{-c}$$
 \end{lemma}
\begin{lemma}
Let $0<p<1$. If $\vert g(z)\vert^2 \left(1-\vert z\vert^2\right)^p dA(z)$ is a $p$-Carleson measure, then $\vert g(z)\vert dA(z)$ is a Carleson measure.
\end{lemma}
\begin{proof}
One can obtain this result using Cauchy-Schwarz inequalty and Lemma 2 (see, for example, [P, Lemma 2.2]).
\end{proof} 

We also need the next result, which can be found in [P] ( see also Theorem 7.4.2 of [X2]).
\begin{lemma} 
Let $p>0$. Then $g\in\mathcal{M}(Q_p)$ if and only if  $g\in H^{\infty}(\mathbb{D})$, and for all $f\in {Q_p}$, the measure $\vert f(z)\vert^2 \vert g'(z)\vert ^2 \left(1-\vert z\vert^2\right)^p dA(z)$ is a $p$-Carleson measure. 
\end{lemma}
\smallskip
\subsection{Boundary Values of $Q_p$ Spaces}\  
Let $0<p<1$.  A function  $f\in L^2(\mathbb {T})$ is said to be in $Q_p(\mathbb{T})$ if 
$$ \vert \vert f\vert\vert _{Q_p (\mathbb{T})}^2:=\int_{-\pi}^\pi \vert f \vert ^2 d\sigma + \sup_{I\subset\mathbb{T}}\frac{1}{\vert I\vert^p} \int_I \int_I \frac{\vert f(\xi)-f(\eta)\vert ^2}{\vert \xi -\eta \vert ^{2-p}} \vert d\xi\vert \vert d\eta \vert <\infty, $$ where the supremum is taken over all arcs $I\subset\mathbb{T}.$\\
Since $Q_p \subset H^2$, where $H^2$ is the classical Hardy space, any function $f\in Q_p$ has a non-tangential radial limit almost everywhere on $\mathbb{T}$. It is also true that, for $p\in (0,1), Q_{p}=Q_{p}(\mathbb{T})\cap{H}^{2}(\mathbb{D})$. The following Lemma from  [X2] proves that an analytic function $f$ on $\mathbb{D}$ is in $Q_p$ if and only if its boundary values lie on $Q_p(\mathbb{T}).$

\begin{lemma}
Let $p\in (0,1)$ and let $f\in H^2$. Then $f\in Q_p$ if and only if 
$$\Vert f\Vert_{Q_p(\mathbb{T})}^2=\int_{-\pi}^\pi \vert f \vert ^2 d\sigma +\sup_{I\subset T} \vert I\vert^{-p} \int_I \int_I \frac{\vert f(\zeta)-f(\eta)\vert^2}{\vert \zeta -\eta \vert ^{2-p}} \vert d\zeta \vert \vert d \eta \vert <\infty.$$
\end{lemma}
We can see in the proof that $$\Vert f\Vert_{Q_p}\approx \Vert f\Vert _{Q_p(\mathbb{T})}.$$
The next Lemma, proved in [NX] (see also [X2, Corollary 7.1.1]), is also an important tool for us to check that a function belongs to $Q_p(\mathbb{T}).$
\begin{lemma} 
Let $\; 0<p<1$ and $f\in L^2(\mathbb{T})$, and let $ F\in C^{1}(\mathbb{D})$  such that $\underset{r\rightarrow 1^{-}}\lim F(re^{it})=f(e^{it})$ for a.e. $e^{it}\in \mathbb{T}$. If $\vert\nabla  F(z)\vert ^2 \left(1-\vert z\vert ^2\right)^p dA(z)$ is a $p$-Carleson measure, then $f\in Q_p(\mathbb{T}).$
\end{lemma}
With the help of this lemma, we can easily see that a function $f\in L^2(\mathbb{T})$ belongs to $Q_p(\mathbb{T})$ if and only if $\vert\nabla  \tilde{f}(z)\vert ^2 \left(1-\vert z\vert ^2\right)^p dA(z)$ is a $p$-Carleson measure, where $\tilde{f}$ denotes the Poisson integral of $f$(for example, see [NX]).

Let  $\mathcal{M}(Q_{p}(\mathbb{T}))$  be the space of multipliers on $Q_p(\mathbb{T}),$ that is  $$\mathcal{M}(Q_{p}(\mathbb{T})):=\left\{ \phi\in Q_{p}(\mathbb{T}) :\, M_{\phi}(f)=\phi f\in Q_{p}(\mathbb{T})\; \text{ for all } f\in Q_{p}(\mathbb{T})\right\}.$$
As in $Q_p$, it's clear that  $\mathcal{M}(Q_{p}(\mathbb{T}))\subseteq L^\infty \cap Q_p(\mathbb{T})$. 
One of the important parts of the proofs of our Theorems 1 and 2 is establishing the  solvability of the  $\bar{\partial}$ - equation, which turns out to be simpler because of the following Lemma.

\begin{lemma}
Let $p\in (0,1).$ If $d\lambda(z)=\vert g(z)\vert ^2\left(1-\vert z\vert ^2\right)^{p}dA(z)$ is a $p$-Carleson measure, then $\bar{\partial}u=g$ has a solution $v\in Q_{p}(\mathbb{T})\cap L^{\infty}(\mathbb{T})$
such that
\begin{align*}
 (i) &\;\;   \vert\vert v \vert \vert_{Q_p(\mathbb{T})}+\vert \vert v \vert \vert _{L^{\infty}(\mathbb{T})}\leq C \vert \vert \lambda \vert \vert _{p}^{\frac{1}{2}}, \; \text {where C is an absolute constant}.\\
(ii) & \;\; vf\in Q_p(\mathbb{T})\;\;  \text{for all}\;\; f\in Q_p.
\end{align*}  
\end{lemma}
\begin{proof}
Reader can find the proof of (i) in [NX]. The solution $v$ taken there was
\begin{align*}
 v(z)=&\frac{i}{\pi} \int_D \frac{1-\vert \zeta \vert ^2}{\vert 1-\bar{\zeta} z\vert ^2}.\\
 &  \exp \left(\int_{ \vert w\vert \geq \vert \zeta\vert} \left( \frac{1+\bar{w} \zeta}{1-\bar{w} \zeta}-\frac{1+\bar{w} z}{1-\bar{w}z}\right)\vert g(w)\vert dA(w)\right) \vert g(\zeta)\vert dA(\zeta),
\end{align*} 
that has the same boundary values of   $zu(z)$.\\

For (ii), we need to show that $vf\in Q_p(\mathbb{T})\;\;  \text{for all}\;\; f\in Q_p$. To prove that $vf\in Q_p(\mathbb{T})$, by Lemma 6, it's enough to show that $$\vert \nabla (vf)(z)\vert ^2 (1-\vert z\vert ^2)^p dA(z)$$ is a $p$-Carleson measure.\\
Since $f\in Q_p$ and $v\in L^{\infty}(\mathbb{D})$, we have that $$\vert v(z)\vert^2 \vert \nabla f(z)\vert ^2 (1-\vert z\vert ^2)^p dA(z)$$ is a $p$-Carleson measure. So it remains to show that $$\vert f(z)\vert^2 \vert \nabla v(z)\vert ^2 (1-\vert z\vert ^2)^p dA(z)$$ is a $p$-Carleson measure.   Since $\vert g(z)\vert ^2\left(1-\vert z\vert ^2\right)^{p}dA(z)$ is a $p$-Carleson measure, we can see in the proof of Theorem 1 in  [P] that  $$\vert f(z)\vert^2 \vert \nabla v(z)\vert ^2 (1-\vert z\vert ^2)^p dA(z)$$ is a $p$- Carleson measure.\\ 
This completes the proof of Lemma 7.
\end{proof}
A similar result for  $1$-Carleson measure was proved in [G, P. 320-322].\\
\section{Proof of  Theorems }
First, by using the normal families, we will assume that the given family of functions are analytic in some neighborhood of $\overline{\mathbb{D}}$ and then reduce our theorems to the problem of solving certain inhomogoneous Cauchy-Riemann equations. Doing that will allow us to find  smooth solutions for both (4) and (6). Then we will convert our obtained smooth solutions into $H^{\infty}(\mathbb{D})\cap Q_p$ and  $\mathcal{M}(Q_p)$ - solutions, using some correction functions and applying the size conditions (3) and (5), respectively.
\subsection{Proof of Theorem 1}
Let, $ f_1, ..., f_n, g\in H^{\infty}(\mathbb{D})\cap Q_p$ such that they satisfy (3). Also, suppose  that $g, f_1,...f_n $ are analytic on $\overline{\mathbb{D}}$. Moreover, we assume $\vert \vert f_{j}\vert \vert \leq 1, \vert \vert g\vert \vert \leq 1.$
\smallskip

Set $$ \psi_{j}=\frac{g\; \overline{f_j}}{\sum_{l=1}^n |f_l|^2},  \;\; j=1,2,...,n.$$
 
Then, using (3), $\vert \psi_{j}\vert \leq 1$, and $C^{\infty}$ on $\overline{\mathbb{D}}$ and $$ \psi_{1}f_1+...+\psi_{n}f_n=g.$$
Suppose we can solve 
\begin{equation}
\frac{\partial{b_{j,k}}}{\partial { \overline{z}}}=g\psi_{j}\frac{\partial{\psi_{k}}}{\partial { \overline{z}}}=g^{3}G_{j,k}(z),\;\; 1\leq j ,k \leq n,
\end{equation}
with
\begin{equation}
b_{j,k}  \in L^{\infty}(\mathbb{T}) \cap Q_{p} (\mathbb{T})  .
\end{equation}
The difficulty, of course, is that $\psi_j(z)$ may not be analytic on $\mathbb{D}$. To rectify that, we write

$$g_j(z)=g^{2}(z)\psi_j(z)+\sum_{k=1}^n\left( b_{j,k}(z)-b_{k,j}(z) \right)f_k(z),$$
 then we get $$\sum_{j=1}^{n}g_{j}f_{j}=g^{2}\sum_{j=1}^{n} \psi_{j}f_{j}=g^3.$$ and also

$$\frac{\partial g_j}{\partial \bar{z}}=0.$$
Provided the solution of (7) satisfying (8), the functions $g_{j}$ are bounded analytic solutions for (4). Now, we will try to show that $g_{j}\in Q_p$. 

Since, $\frac{\partial \overline{f_l}}{\partial \bar{z}}=\overline{f'_l},$

$$\frac{\partial \psi_{j}}{\partial \bar{z}} =\frac{g\bar{f'_j}}{\sum_{l=1}^n \, |f_l|^2}-\frac{g\bar{f_{j}}\sum_{l=1}^{n}{f_{l}\bar{f'_{l}}}}{\left(\sum_{l=1}^n \,\vert{ f_l}\vert^2 \right)^2}
=\frac{g \sum_{l=1}^n f_{l}\left(\bar{f_l}\bar{f'_j}-\bar{f_k}\bar{f'_l}\right)}{\left(\sum_{l=1}^n \,\vert{ f_l}\vert^2 \right)^2}.$$ 
Thus, 
\begin{align*}
\left\vert \frac{\partial \psi_{j}}{\partial \bar{z}} \right\vert^2 \lesssim \frac { 2 \vert g\vert^2 \left(\sum \vert f_l \vert ^2\right)^2 \sum \vert f'_l \vert ^2}{\left(\sum_{l=1}^n \,\vert{ f_l}\vert^2 \right)^4}=&\frac { 2 \vert g\vert^2  \sum \vert f'_l \vert ^2}{\left(\sum_{l=1}^n \,\vert{ f_l}\vert^2 \right)^2}\\
& \leq \frac{2\sum \vert f'_l \vert ^2}{\sum_{l=1}^n \,\vert{ f_l}\vert^2}.  \;\; \left(\text{using (3)}\right)
\end{align*}

Similarly,
$$ \frac{\partial \psi_{j}}{\partial{z}}=\frac{g' \bar{f_j}}{\sum_{l=1}^n \, |f_l|^2}-\frac{g\bar{f_{j}}\sum_{l=1}^{n}{f'_{l}\bar{f_{l}}}}{\left(\sum_{l=1}^n \,\vert{ f_l}\vert^2 \right)^2}.$$
Therefore,
\begin{align*}
\left\vert \frac{\partial \psi_{j}}{\partial{z}}\right\vert^2& \lesssim \frac{\vert g'\vert^{2}\sum \vert f_{l} \vert^2}{\left(\sum_{l=1}^n \, \vert f_l\vert\right)^2 }+\frac{\vert g\vert^2 \left(\sum \vert f_l\vert^2 \right)^2 \sum \vert f'_l \vert ^2}{\left(\sum_{l=1}^n \,\vert{ f_l}\vert^2 \right)^4}\\
&\lesssim \frac{\vert g'\vert^2 }{\sum_{l=1}^n \, |f_l|^2}+\frac{\sum \vert f'_l \vert ^2}{\sum_{l=1}^n \,\vert{ f_l}\vert^2 }.
\end{align*}
Hence,
\begin{align*}
\left \vert \nabla \psi_j \right\vert^2&=2\left\vert \frac{\partial \psi_{j}}{\partial{z}}\right\vert^2+2\left\vert \frac{\partial \psi_{j}}{\partial \bar{z}}\right\vert^2\\
&\lesssim \frac{\vert g'\vert^2 }{\sum_{l=1}^n \, |f_l|^2}+\frac{\sum \vert f'_l \vert ^2}{\sum_{l=1}^n \,\vert{ f_l}\vert^2}.
\end{align*} 
Applying the size condition (3), we get that   $$\left \vert g^2 \;\; \nabla \psi_{j} \right\vert^2 \lesssim \vert g'\vert^2  \, +\sum {\vert f'_l \vert ^2 }.$$

Also, since $\psi_{j}$ are $C^{\infty}$ on $\overline{\mathbb{D}}$, by Lemma 6, to see that $g^{2}\psi_{j} \in Q_{p}(\mathbb{T})$, it is enough to show that $\left\vert \nabla (g^{2} \psi_{j} )(z)\right\vert^2 \left(1-\vert z \vert ^{2}\right)^{p}$ is a $p$-Carleson measure. Using the fact  $g, f_{k} \in Q_p$ , we get that for any Carleson box $S(I)$ on any interval $I$ of $\mathbb{T}$, 
\begin{align*}
\int_{S(I)}\vert  \psi _{j}(z)\nabla (g^2(z))\vert^2\left(1-\vert z\vert ^2\right)^{p}dA(z)& \lesssim \int _{S(I)}  \vert g'(z)\vert^2 \left(1-\vert z\vert^2\right)^{p} dA(z)\\
& \lesssim\vert I\vert^{p}.
\end{align*}
 And
\begin{align*}
\int_{S(I)}\vert  g^2(z) \nabla \psi _{j}(z)\vert^2\left(1-\vert z\vert ^2\right)^{p}dA(z)& \lesssim \int _{S(I)}  \vert g'(z)\vert^2 \left(1-\vert z\vert^2\right)^{p} dA(z)\\ & +\sum_{l=1}^n \int _{S(I)} \vert f'_{j}(z)\vert^2 \left(1-\vert z\vert^2\right)^{p}dA(z)\\
& \lesssim\vert I\vert^{p}.
\end{align*} 
Thus, $g^{2} \psi_{j}\in Q_{p}\left(\mathbb{T}\right).$ \
Also since $ f_1,...,f_n\in H^{\infty} \cap Q_p$, by (8) we find that $$\sum_{k=1}^n \left(b_{j,k}-b_{k,j}\right)f_{k}\in Q_{p}\left(\mathbb{T}\right).$$
This implies that $g_{j}f\in Q_{p}(\mathbb{T})$.\\
 Therefore,  the functions $g_j$ are bounded analytic whose boundary values lies on $Q_{p}(\mathbb{T})$. Hence, $g_j$ are the required solutions of (4).
  
Looking back at the above proof, we will be done if we can find a solution of $\bar{\partial}\;b_{j,k}=g\psi_{j}\bar{\partial}{\psi_k}$ satisfying (8). For this, it is enough to deal with an equation $\bar{\partial}(u)=G$, where $G=g\;\psi_j \overline\partial{\psi_k}$.\\
We have, 
\begin{align*}
\vert G\vert^2& \leq \vert g \vert^2 \; \vert \psi_{j} \vert^2 \;  \vert \overline\partial{\psi_k}\vert^2\\
& \lesssim \frac{\vert g\vert^2\;\sum \vert f'_l \vert ^2}{\sum_{l=1}^n \,\vert{ f_l}\vert^2}\leq \sum \vert f'_l \vert ^2.
\end{align*} \\

Since $f_l\in Q_p$, by Lemma 1, $\vert f'_l(z)\vert ^2 \left(1-\vert z\vert^2\right)^p dA(z)$ is a $p$-Carleson measure. Therefore, $\vert G(z)\vert ^2\left(1-\vert z\vert^2\right)^p dA(z)$ is a $p$ - Carleson measure.

 Hence, using Lemma 7, we obtain a solution $v\in Q_{p}(\mathbb{T})\cap L^{\infty}(\mathbb{T})$ of  $\bar{\partial}u=G$.
This completes the proof of Theorem 1.\\

To prove Theorem 2, we use arguments similar to those we used in Theorem 1, but the difference is finding the solutions in $\mathcal{M}(Q_p)$ for the given data on $\mathcal{M}(Q_p)$.

\subsection{Proof of Theorem 2}
Let, $ f_1, ..., f_n, g\in \mathcal{M}(Q_p)$ such that they satisfy (5) and  are analytic on $\overline{\mathbb{D}}$. Moreover, we assume $\vert \vert f_{j}\vert \vert \leq 1, \vert \vert g\vert \vert \leq 1.$
\medskip\

In this case, taking the $\psi_{j}$ as in Theorem 1, we will 
suppose that we can solve 
\begin{equation}
\frac{\partial{b_{j,k}}}{\partial { \overline{z}}}=g\psi_{j}\frac{\partial{\psi_{k}}}{\partial { \overline{z}}}=g^{3}G_{j,k}(z),\;\; 1\leq j ,k \leq n,
\end{equation}
satisfying 
\begin{equation}
b_{j,k}  \in L^{\infty}(\mathbb{T}) \cap Q_{p} (\mathbb{T})\;\;  \text{with}\;\;b_{j,k}f  \in Q_{p} (\mathbb{T}) \;\; \text{for all}\;\; f\in Q_p.
\end{equation}\\
Hence, 
$$g_j(z)=g^{2}(z)\psi_j(z)+\sum_{k=1}^n\left( b_{j,k}(z)-b_{k,j}(z) \right)f_k(z)$$ are bounded analytic on $\mathbb{D}$ and also  $$\sum_{j=1}^{n}g_{j}f_{j}=g^{2}\sum_{j=1}^{n} \psi_{j}f_{j}=g^3.$$ 
So the functions $g_{j}$  are analytic solutions for (6). Now, we will try to show that $g_{j}\in \mathcal{M}\left(Q_p\right)$. For this, our aim is to show that ${g_{j}f}\in Q_{p} $
for all $f\in Q_{p}.$ \\ But,  $Q_{p}=Q_{p}(\mathbb{T})\cap{H}^{2}(\mathbb{D})$, it's sufficient to show that $g_{j} f \in Q_{p}(\mathbb{T})$ for all $f\in Q_{p}$ .

Also, since $\psi_{j}$  are $C^{\infty}$ on $\overline{\mathbb{D}}$, by Lemma 6, to see that $g^{2}\psi_{j} f \in Q_{p}(\mathbb{T})$, it is enough to show that $\left\vert \nabla (g^{2} \psi_{j}f )(z)\right\vert^2 \left(1-\vert z \vert ^{2}\right)^{p}$ is a $p$ - Carleson measure.\\

Using the facts that $g, f_{k} \in \mathcal{M}(Q_p)$, 

$$\left \vert \nabla \psi_j \right\vert^2
\lesssim \frac{\vert g'\vert^2 }{\sum_{l=1}^n \, |f_l|^2}+\frac{\sum \vert f'_l \vert ^2}{\sum_{l=1}^n \,\vert{ f_l}\vert^2}$$ and the size condition (5),  for any Carleson box  $S(I)$ based on any interval $I$ of $\mathbb{T}$, we have that
\begin{align*}
\int_{S(I)}\vert {g^2(z)} \psi _{j}(z)\nabla f(z)\vert^2\left(1-\vert z\vert ^2\right)^{p}dA(z)& \lesssim \int _{S(I)} \vert f'(z)\vert^2 \left(1-\vert z\vert^2\right)^{p} dA(z)\\
& \lesssim\vert I\vert^{p}.
\end{align*}

\begin{align*}
\int_{S(I)}\vert {f(z)} \psi _{j}(z)\nabla g^2(z)\vert^2\left(1-\vert z\vert ^2\right)^{p}dA(z)& \lesssim \int _{S(I)} \vert f(z)\vert^2 \vert g'(z)\vert^2 \left(1-\vert z\vert^2\right)^{p} dA(z)\\
& \lesssim\vert I\vert^{p}.
\end{align*}
 And
\begin{align*}
\int_{S(I)}\vert {f(z)} \nabla \psi _{j}(z) g^2(z)\vert^2\left(1-\vert z\vert ^2\right)^{p}dA(z)& \lesssim \int _{S(I)} \vert f(z)\vert^2 \vert g'(z)\vert^2 \left(1-\vert z\vert^2\right)^{p} dA(z)\\ & +\sum_{l=1}^n \int _{S(I)} \vert f(z)\vert^2 \vert f'_{j}(z)\vert^2 \left(1-\vert z\vert^2\right)^{p}dA(z)\\
& \lesssim\vert I\vert^{p}.
\end{align*} 
Therefore, $(g^{2} \psi_{j})f\in Q_{p}\left(\mathbb{T}\right)$.\\

Also, since $ f_1,...f_n\in \mathcal{M}\left(Q_p\right)$, by (10) we find that $$\sum_{k=1}^n \left(b_{j,k}-b_{k,j}\right)f_{k}f\in Q_{p}\left(\mathbb{T}\right).$$
This implies that $g_{j}f\in Q_{p}$. Therefore the functions $g_j$ are in $\mathcal{M}\left(Q_p\right)$.
\medskip\
Now, it remains to show that we can find a solution of (9) satisfying (10). For this, we have that  
 $$G(z)=g\;\psi_j \overline\partial{\psi_k},$$ which satisfies 
\begin{align*}
\vert G\vert^2& \leq \vert g \vert^2 \; \vert \psi_{j} \vert^2 \;  \vert \overline\partial{\psi_k}\vert^2\\
& \lesssim \frac{\vert g\vert^2\;\sum \vert f'_l \vert ^2}{\sum_{l=1}^n \,\vert{ f_l}\vert^2}\leq \sum \vert f'_l \vert ^2.
\end{align*} 
Hence, $\vert G(z)\vert ^2 (1-\vert z\vert^2 )^p$ is a $p$-Carleson measure. Now, applying the second part of  Lemma 7, we get the required solution of (9). \\
This completes the proof of Theorem 2.

\section{Remarks and Questions}
\subsection{Wolff's Theorem for Infinite Number of  Generators}\
Since we took only the finite number of generators, we were able to show that $\vert G(z)\vert^2 \left(1-\vert z\vert^2\right)^p dA(z)$ is a $p$-Carleson measure. But, we do not know up to this point whether this can be done  taking infinite number of generators, because the p-Carleson constant may depend on n. It would be interesting to see if one can generalize both Theorems 1 and  2 for any number of generators.\
 \subsection{Generalized Ideal Problem on $Q_p$ Spaces}
 We proved  that the size conditions (3) and (5) imply $g^3 \in \mathcal{I}(f_1,...,f_n)$ in $ H^{\infty}(\mathbb{D})\cap Q_p$  and in $\mathcal{M}(Q_p)$. Several authors (for example, Treil[T2], Trent[Tr1],..) have given sufficient conditions for $g\in \mathcal{I}(f_1,...,f_n)$ in $H^{\infty}(\mathbb{D})$. It opens up the problem of whether a single sufficient condition can generalize the ideal problem.     \\

\medskip
Acknowledgement:  The author would like to thank the referee for the careful review and providing important comments. The author also would like to  thank Dr. J. Xiao, Dr. J. Pau, Dr. A. Nicolau, and Dr. Z. Wu for having  useful conversations on different properties of $Q_p$ spaces.

\end{document}